\documentclass[12pt,a4paper]{article}
\usepackage{mathrsfs}
\usepackage{epsfig, graphicx}
\usepackage{latexsym,amsfonts,amsbsy,amssymb}
\usepackage{amsmath,amsthm}
\usepackage{color}
\usepackage[colorlinks, citecolor=blue]{hyperref}
\makeatletter % '@' is now a normal "letter" for TeX

\@addtoreset{equation}{section}
\makeatother % '@' is restored as a "non-letter" character for TeX
\allowdisplaybreaks % For long formula

\newtheorem{theorem}{Theorem}[section]

\newtheorem{remark}{Remark}[section]
\newtheorem{algorithm}{Algorithm}[section]

\newtheorem{proposition}{Proposition}[section]
\newtheorem{example}{Example}[section]

\begin{document}
\title{Local and Parallel Finite Element Algorithm Based On Multilevel Discretization
for Eigenvalue Problem\footnote{This work was supported in part by National Science Foundations of China
 (NSFC 11001259, 11371026, 11201501, 11031006,
2011CB309703),  the National Center for Mathematics and Interdisciplinary Science,
 CAS and the President Foundation of AMSS-CAS.}}
\author{
Yu Li\footnote{LSEC, ICMSEC,
Academy of Mathematics and Systems Science, Chinese Academy of
Sciences, Beijing 100190, China (liyu@lsec.cc.ac.cn)},\ \
Xiaole Han\footnote{LSEC, ICMSEC,
Academy of Mathematics and Systems Science, Chinese Academy of
Sciences, Beijing 100190, China (hanxiaole@lsec.cc.ac.cn)}, \ \
Hehu Xie\footnote{LSEC, ICMSEC,
Academy of Mathematics and Systems Science, Chinese Academy of
Sciences, Beijing 100190, China (hhxie@lsec.cc.ac.cn)} \ \ and\ \
Chunguang You\footnote{LSEC, ICMSEC,
Academy of Mathematics and Systems Science, Chinese Academy of
Sciences, Beijing 100190, China (youchg@lsec.cc.ac.cn)}
}
\date{}
\maketitle
\begin{abstract}
A local and parallel algorithm based on the multilevel discretization is proposed in this paper to solve the
eigenvalue problem by the finite element method. With this new scheme, solving the eigenvalue problem
in the finest grid is transferred to solutions of the eigenvalue problems on the coarsest mesh and a
series of solutions of boundary value problems by using the local and parallel algorithm.
 The computational work in each processor can reach the optimal order.
Therefore, this type of multilevel local and parallel method improves the overall efficiency of solving
the eigenvalue problem. Some numerical experiments are presented to validate the efficiency
of the new method.

\vskip0.3cm {\bf Keywords.} eigenvalue problem, multigrid, multilevel correction,
local and parallel method, finite element method.

\vskip0.2cm {\bf AMS subject classifications.} 65N30, 65N25, 65L15, 65B99.
\end{abstract}

\section{Introduction}

Solving large scale eigenvalue problems becomes a fundamental problem in modern science and engineering society.
However, it is always a very difficult task to solve high-dimensional eigenvalue problems which come from
physical and chemistry sciences.  Xu and Zhou \cite{XuZhou_Eigen} give a type of two-grid discretization method
to improve the efficiency of the solution of eigenvalue problems. By the two-grid method,
the solution of eigenvalue problem on a fine mesh is reduced to a solution of
eigenvalue problem on a coarse mesh (depends on the fine mesh) and a
solution of the corresponding boundary value problem on the fine mesh \cite{XuZhou_Eigen}.
For more details, please read \cite{Xu_Two_Grid,Xu_Nonlinear}.
Combing the two-grid idea and the local and parallel finite element
technique \cite{XuZhou_FEM},  a type of local and parallel finite element technique
 to solve the eigenvalue problems is given in \cite{XuZhou_Parallel} (also see \cite{DaiShenZhou}).
 Recently, a new type of multilevel correction method for solving eigenvalue problems
  which can be implemented on multilevel grids is proposed in \cite{LinXie}.
   In the multilevel correction scheme, the solution of eigenvalue problem
on a finest mesh can be reduced to a series of solutions of the eigenvalue problem on a very coarse
 mesh (independent of finest mesh) and a series of solutions of the boundary value problems
 on the multilevel meshes.  The multilevel correction method gives a way to
  construct a type of multigrid scheme for the eigenvalue problem \cite{LinXie_MultiGrid_Eigenvalue}.

In this paper, we propose a type of multilevel local and parallel scheme to solve the
eigenvalue problem based on the combination of
the multilevel correction method and the local and parallel technique. An special property of
this scheme is that we can do the local and parallel computing for any times and then
the mesh size of original coarse triangulation is independent of the finest triangulation.
With this new  method, the solution of the eigenvalue problem will not be more difficult
than the solution of the boundary value problems by the local and parallel algorithm
since the main part of the computation in the multilevel local and parallel method
 is solving the boundary value problems.

The standard Galerkin finite element method for eigenvalue problem
has been extensively investigated, e.g. Babu\v{s}ka and Osborn
\cite{Babuska2,BabuskaOsborn}, Chatelin \cite{Chatelin} and
references cited therein. There also exists analysis for the local and parallel
finite element method for the boundary value problems and eigenvalue problems
\cite{DaiShenZhou,SchatzWahlbin,Wahlbin,XuZhou_FEM,XuZhou_Eigen,XuZhou_Parallel}.
 Here we adopt some basic results in these papers for our analysis.
The corresponding error and computational work estimates of the proposed multilevel
local and parallel scheme for the eigenvalue problem will be analyzed. Based
on the analysis, the new method can obtain optimal errors with an optimal
computational work in each processor.
%The proposed multilevel local and parallel finite element algorithm can be described as follows:
%(I)\ solve the eigenvalue problem in the original coarse finite element space;
%(II)\ solve an additional boundary value problem in each subdomain on the refined mesh using
%the previous obtained eigenvalue multiplying the corresponding
%eigenfunction as the load vector;
%(III)\ solve an eigenvalue problem
%again on the finite element space which is constructed by combining
%the coarsest finite element space with the obtained eigenfunction
%approximation in step (II). Then refine mesh and go to step (II) for next loop until stop.

An outline of this paper goes as follows. In the next section,
 a basic theory about the local error estimate of the finite element method is introduced.
In Section 3, we introduce the finite element method for the eigenvalue problem
and the corresponding error estimates.
A local and parallel type of one correction step and multilevel correction algorithm will be given in Section 4.
The estimate of the computational work for the multilevel local and parallel algorithm
is presented in section 5. In Section 6, two numerical examples are presented
 to validate our theoretical analysis and some concluding remarks are given in the last section.

\section{Discretization by finite element method}

In this section, we introduce some notation and error estimates of
the finite element approximation for linear elliptic problem.
The letter $C$ (with or without subscripts) denotes a generic
positive constant which may be different at its different occurrences through the paper.
For convenience, the symbols $\lesssim$, $\gtrsim$ and $\approx$
will be used in this paper. That $x_1\lesssim y_1, x_2\gtrsim y_2$
and $x_3\approx y_3$, mean that $x_1\leq C_1y_1$, $x_2 \geq c_2y_2$
and $c_3x_3\leq y_3\leq C_3x_3$ for some constants $C_1, c_2, c_3$
and $C_3$ that are independent of mesh sizes (see, e.g., \cite{Xu}).
We shall use the standard notation for Sobolev spaces $W^{s,p}(\Omega)$ and their
associated norms and seminorms (see, e.g., \cite{Adams}). For $p=2$, we denote
$H^s(\Omega)=W^{s,2}(\Omega)$ and $H_0^1(\Omega)=\{v\in H^1(\Omega):\ v|_{\partial\Omega}=0\}$,
where $v|_{\partial\Omega}=0$ is in the sense of trace, $\|\cdot\|_{s,\Omega}=\|\cdot\|_{s,2,\Omega}$.

For $G\subset D\subset \Omega$, the notation $G\subset\subset D$ means
 that ${\rm dist}(\partial D\setminus\partial\Omega,\partial G\setminus\partial\Omega)>0$
  (see Figure \ref{fig:DomainDG}).  It is well known that
 any $w\in H_0^1(\Omega_0)$ can be naturally extended to be a function in $H_0^1(\Omega)$ with zero
outside of $\Omega_0$, where $\Omega_0\subset\Omega$.
Thus we will show this fact by the abused notation $H_0^1(\Omega_0)\subset H_0^1(\Omega)$.
\begin{figure}[htb]
\centering
\includegraphics[width=10cm,height=4.5cm]{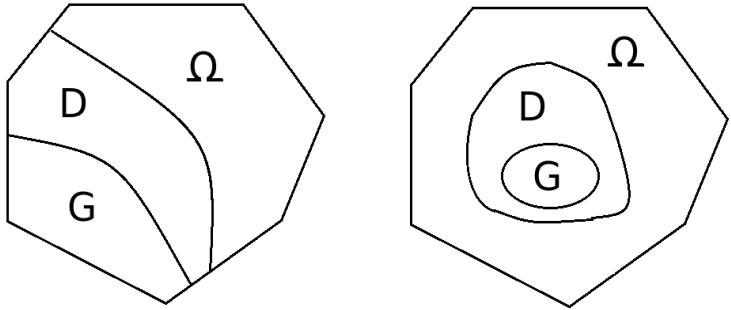}
\caption{$G\subset\subset D\subset\subset\Omega$}\label{fig:DomainDG}
\end{figure}

\subsection{Finite element space}

Now, let us define the finite element space.
First we generate a shape-regular
decomposition $\mathcal{T}_h(\Omega)$ of the computing domain $\Omega\subset \mathcal{R}^d\
(d=2,3)$ into triangles or rectangles for $d=2$ (tetrahedrons or
hexahedrons for $d=3$). The diameter of a cell $K\in\mathcal{T}_h(\Omega)$
is denoted by $h_K$. The mesh size function is denoted by $h(x)$ whose value
is the diameter $h_K$ of the element $K$ including $x$.

For generality, following \cite{XuZhou_FEM,XuZhou_Parallel},
we shall consider a class of finite element spaces that satisfy certain assumptions.
Now we describe such assumptions.

{\bf A.0}.\ There exists $\gamma>1$ such that
\begin{eqnarray*}\label{Mesh_Size_Condition}
h_{\Omega}^{\gamma}\lesssim h(x),\ \ \ \ \forall x\in\Omega,
\end{eqnarray*}
where $h_{\Omega}=\max_{x\in\Omega}h(x)$ is the
largest mesh size of $\mathcal{T}_h(\Omega)$.

Based on the triangulation $\mathcal{T}_h(\Omega)$, we define the finite element space
$V_h(\Omega)$ as follows
\begin{eqnarray*}\label{FES_Pk}
V_h(\Omega)=\big\{v\in C(\bar{\Omega}):\ v|_K\in \mathcal{P}_k,\ \ \forall K\in\mathcal{T}_h(\Omega)\big\},
\end{eqnarray*}
where $\mathcal{P}_k$ denotes the space of polynomials of degree not greater than a positive integer
$k$.  Then we know $V_h(\Omega)\subset H^1(\Omega)$ and define
$V_{0h}(\Omega)=V_h(\Omega)\cap H_0^1(\Omega)$. Given $G\subset \Omega$,
we define $V_h(G)$ and $\mathcal{T}_h(G)$ to be the restriction
of $V_h(\Omega)$ and $\mathcal{T}_h(\Omega)$ to $G$, respectively, and
\begin{eqnarray*}\label{FES_G}
V_{0h}(G)=\big\{v\in V_h(\Omega):\ {\rm supp}v\subset\subset G\big\}.
\end{eqnarray*}
For any $G\subset \Omega$ mentioned in this paper, we assume that it aligns with the partition
$\mathcal{T}_h(\Omega)$.

%We now states our basic assumption on the finite element spaces.

%{\bf A.1}.\ {\it Approximation}.\ If $w\in H_0^1(\Omega)$, then as $h_{\Omega}\rightarrow 0$,
%\begin{eqnarray*}\label{Approximation}
%\inf_{v\in V_{0h}(\Omega)}\big(\|h^{-1}(w-v)\|_{0,\Omega}+\|w-v\|_{1,\Omega}\big)=o(1).
%\end{eqnarray*}
%
%{\bf A.1'}.\ {\it Approximation}.\ There exists $k\geq 1$ such that for $w\in H_0^1(\Omega)$,
%\begin{eqnarray*}\label{Approximation_II}
%\inf_{v\in V_{0h}(\Omega)}\big(\|h^{-1}(w-v)\|_{0,\Omega}+\|w-v\|_{1,\Omega}\big)\lesssim \|h^sw\|_{1+s,\Omega},
%\ \ \ 0\leq s\leq k.
%\end{eqnarray*}
%
%
%{\bf A.2}.\ {Inverse Estimate}.\ For any $v\in V_h(\Omega_0)$,
%\begin{eqnarray*}\label{Inverse_Estimate}
%\|v\|_{1,\Omega_0}\lesssim \|h^{-1}v\|_{0,\Omega_0}.
%\end{eqnarray*}
%
%{\bf A.3}.\ {\it Superapproximation}.\ For $G\subset \Omega_0$, let $\omega\in C_0^{\infty}(\Omega)$
%with ${\rm supp}\omega\subset\subset G$. Then for any $w\in V_h(G)$, there is $v\in V_{0h}(G)$ such that
%\begin{eqnarray*}\label{Superapproximation}
%\|h^{-1}(\omega w-v)\|_{1,G}\lesssim \|w\|_{1,G}.
%\end{eqnarray*}

As we know, the finite element space $V_h$ satisfy the following proposition (see, e.g., \cite{BrennerScott,CiarletLions,XuZhou_FEM,XuZhou_Parallel}).
\begin{proposition}\label{Proposition_Fractional_Norm}({\it Fractional Norm})
For any $G\subset \Omega$, we have
\begin{eqnarray}\label{Fractional_Norm}
\inf_{v\in V_{0h}(G)}\|w-v\|_{1,G}\lesssim \|w\|_{1/2,\partial G},\ \ \ \forall w\in V_h(\Omega).
\end{eqnarray}
\end{proposition}
%
%Now, we give an example of the normal finite element spaces which satisfies the above assumptions.
%For simplicity, we assume that $\Omega$ is a polygonal domain, $\mathcal{T}_h(\Omega)$ consist of
%shape-regular simplices and  $V_h(\Omega)$ is a space of continuous piecewise polynomial
% of degree not greater than $k$, namely
%\begin{eqnarray*}\label{FES_Pk}
%V_h(\Omega)=\big\{v\in C(\bar{\Omega}):\ v|_K\in \mathcal{P}_k,\ \ \forall K\in\mathcal{T}_h(\Omega)\big\},
%\end{eqnarray*}
%where $\mathcal{P}_k$ denotes the space of polynomials of degree not greater than a positive integer
%$k$. This is the so-called Lagrange finite element spaces and it satisfies all the above
%assumptions (see, e.g., \cite{BrennerScott,Ciarlet,XuZhou_FEM,XuZhou_Parallel}).

\subsection{A linear elliptic problem}
In this subsection, we repeat some basic properties of a second order elliptic boundary
value problem and its finite element discretization, which will be used in this paper.
The following results is presented in \cite{SchatzWahlbin,Wahlbin,XuZhou_FEM,XuZhou_Parallel}.

We consider the homogeneous boundary value problem
\begin{equation}\label{Liner_Elliptic}
\left\{
\begin{array}{rcl}
Lu&=&f,\ \ {\rm in}\ \Omega,\\
u&=&0,\ \ {\rm on}\ \partial\Omega.
\end{array}
\right.
\end{equation}
Here the linear second order elliptic operator $L:H_0^{1}(\Omega)\rightarrow H^{-1}(\Omega)$ is define as
\begin{eqnarray*}
Lu=-{\rm div}(A\nabla u),
\end{eqnarray*}
where $A=(a_{ij})_{1\leq i,j\leq d}\in \mathcal{R}^{d\times d}$ is uniformly
positive definite symmetric on $\Omega$ with $a_{ij}\in W^{1,\infty}(\Omega)$.
The weak form for (\ref{Liner_Elliptic}) is as follows:

Find $u\equiv  L^{-1}f\in H_0^1(\Omega)$ such that
\begin{eqnarray}\label{Weak_Linear_Elliptic}
a(u, v) = (f, v), \ \ \ \forall v\in H_0^1(\Omega),
\end{eqnarray}
where $(\cdot,\cdot)$ is the standard inner-product of $L^2(\Omega)$ and
\begin{eqnarray*}
a(u,v)=\big(A\nabla u, \nabla v\big).
\end{eqnarray*}
As we know
\begin{eqnarray*}
\|w\|_{1,\Omega}^2\lesssim a(w,w),\ \ \ \forall w\in H_0^1(\Omega).
\end{eqnarray*}
We assume (c.f. \cite{Grisvard}) that the following regularity estimate holds for
the solution of (\ref{Liner_Elliptic}) or (\ref{Weak_Linear_Elliptic})
\begin{eqnarray*}\label{Regularity_Estimate}
\|u\|_{1+\alpha,\Omega}\lesssim \|f\|_{-1+\alpha,\Omega}
\end{eqnarray*}
for some $\alpha\in (0,1]$ depending on $\Omega$ and the coefficient of $L$.

For some $G\subset \Omega$, we need the following regularity assumption

{\bf R(G)}.\ For any $f\in L^2(G)$, there exists a $w\in H_0^1(G)$ satisfying
\begin{eqnarray*}
a(v,w)=(f,v),\ \ \ \forall v\in H_0^1(G)
\end{eqnarray*}
and
\begin{eqnarray*}
\|u\|_{1+\alpha,G}\lesssim \|f\|_{-1+\alpha,G}.
\end{eqnarray*}

For the analysis, we define the Galerkin-Projection operator $P_h:\ H_0^1(\Omega)\rightarrow V_{0h}(\Omega)$ by
\begin{eqnarray}\label{Projection_Problem}
a(u-P_hu,v)=0,\ \ \ \ \forall v\in V_{0h}(\Omega)
\end{eqnarray}
and apparently
\begin{eqnarray}\label{Projection_Inequality}
\|P_hu\|_{1,\Omega}\lesssim \|u\|_{1,\Omega},\ \ \ \forall u\in H_0^1(\Omega).
\end{eqnarray}
Based on (\ref{Projection_Inequality}), the global priori error estimate
 can be obtained from the approximate properties of the finite dimensional subspace $V_{0h}(\Omega)$
 (cf. \cite{BrennerScott,CiarletLions}). For the following analysis, we introduce the following quantity:
\begin{eqnarray}
\rho_{\Omega}(h)&=&\sup_{f\in L^2(\Omega),\|f\|_{0,\Omega}=1}\inf_{v\in V_{0h}(\Omega)}\|L^{-1}f-v\|_{1,\Omega}.
%\eta_{\Omega}(h)&=&\sup_{f\in L^2(\Omega),\|f\|_{1,\Omega}=1}\inf_{v\in V_{0h}(\Omega)}\|L^{-1}f-v\|_{1,\Omega}.
\end{eqnarray}

Similarly, we can also define  $\rho_G(h)$ if Assumption R(G) holds.

The following results can be found in \cite{BabuskaOsborn,BrennerScott,CiarletLions,XuZhou_Eigen,XuZhou_Parallel}.
\begin{proposition}
\begin{eqnarray*}
\|(I-P_h)L^{-1}f\|_{1,\Omega}&\lesssim&\rho_{\Omega}(h)\|f\|_{0,\Omega},\ \ \ \forall f\in L^2(\Omega),\\
\|u-P_hu\|_{0,\Omega}&\lesssim&\rho_{\Omega}(h)\|u-P_hu\|_{1,\Omega},\ \ \ \forall u\in H_0^1(\Omega).
\end{eqnarray*}
\end{proposition}

Now, we state an important and useful result about the local error estimates \cite{SchatzWahlbin,Wahlbin,XuZhou_Parallel}
which will be used in the following.
\begin{proposition}\label{Prop:Local_Estimate}
Suppose that $f\in H^{-1}(\Omega)$ and $G\subset\subset \Omega_0\subset\Omega$. If Assumptions
A.0 holds and $w\in V_{h}(\Omega_0)$ satisfies
\begin{eqnarray*}
a(w,v)=(f,v),\ \ \ \ \forall v\in V_{0h}(\Omega_0).
\end{eqnarray*}
Then we have the following estimate
\begin{eqnarray*}\label{Local_Estimate}
\|w\|_{1,G}\lesssim \|w\|_{0,\Omega_0}+\|f\|_{-1,\Omega_0}.
\end{eqnarray*}
\end{proposition}

\section{Error estimates for eigenvalue problems}
In this section, we introduce the concerned eigenvalue problem and the corresponding
finite element discretization.

In this paper, we consider the following eigenvalue problem:

Find $(\lambda, u )\in \mathcal{R}\times H^1_0(\Omega)$ such that
$b(u,u)=1$ and
\begin{eqnarray}\label{Weak_Eigenvalue_Problem}
  a(u,v)&=&\lambda b(u,v),\quad \forall v\in H^1_0(\Omega),
\end{eqnarray}
where
\begin{eqnarray*}
b(u,u)=(u,u).
\end{eqnarray*}

For the eigenvalue $\lambda$, there exists the following Rayleigh
quotient expression (see, e.g., \cite{Babuska2,BabuskaOsborn,XuZhou_Eigen})
\begin{eqnarray*}\label{Rayleigh_Quotient}
\lambda=\frac{a(u,u)}{b(u,u)}.
\end{eqnarray*}
From \cite{BabuskaOsborn,Chatelin}, we know the eigenvalue problem
(\ref{Weak_Eigenvalue_Problem}) has an eigenvalue sequence $\{\lambda_j \}:$
$$0<\lambda_1\leq \lambda_2\leq\cdots\leq\lambda_k\leq\cdots,\ \ \
\lim_{k\rightarrow\infty}\lambda_k=\infty,$$ and the associated
eigenfunctions
$$u_1,u_2,\cdots,u_k,\cdots,$$
where $b(u_i,u_j)=\delta_{ij}$. In the sequence $\{\lambda_j\}$, the
$\lambda_j$ are repeated according to their multiplicity.

Then we can define the discrete approximation for the exact eigenpair $(\lambda,u)$ of
(\ref{Weak_Eigenvalue_Problem}) based on the finite element space as:

Find $(\lambda_h, u_h)\in \mathcal{R}\times V_{0h}(\Omega)$ such that
$b(u_h,u_h)=1$ and
\begin{eqnarray}\label{Discrete_Weak_Eigen_Problem}
a(u_h,v_h)&=&\lambda_hb(u_h,v_h),\quad \forall v_h\in V_{0h}(\Omega).
\end{eqnarray}

From (\ref{Discrete_Weak_Eigen_Problem}), we know the following
Rayleigh quotient expression for $\lambda_h$ holds
(see, e.g., \cite{Babuska2,BabuskaOsborn,XuZhou_Eigen})
\begin{eqnarray*}\label{Discrete_Rayleigh_Quotient}
\lambda_h &=&\frac{a(u_h,u_h)}{b(u_h,u_h)}.
\end{eqnarray*}
Similarly, we know from \cite{BabuskaOsborn,Chatelin} the eigenvalue
problem (\ref{Discrete_Weak_Eigen_Problem}) has eigenvalues
$$0<\lambda_{1,h}\leq \lambda_{2,h}\leq\cdots\leq \lambda_{k,h}\leq\cdots\leq \lambda_{N_h,h},$$
and the corresponding eigenfunctions
$$u_{1,h}, u_{2,h},\cdots, u_{k,h}, \cdots, u_{N_h,h},$$
where $b(u_{i,h},u_{j,h})=\delta_{ij}, 1\leq i,j\leq N_h$ ($N_h$ is
the dimension of the finite element space $V_{0h}(\Omega)$).

From the minimum-maximum principle (see, e.g., \cite{Babuska2,BabuskaOsborn}),
the following upper bound result holds
$$\lambda_i\leq \lambda_{i,h}, \ \ \ i=1,2,\cdots, N_h.$$

Let $M(\lambda_i)$ denote the eigenspace corresponding to the
eigenvalue $\lambda_i$ which is defined by
\begin{eqnarray}
M(\lambda_i)&=&\big\{w\in V: w\ {\rm is\ an\ eigenvalue\ of\
(\ref{Weak_Eigenvalue_Problem})\ corresponding\ to} \ \lambda_i\nonumber\\
&&\ \ \ {\rm and}\ \|w\|_b=1\big\},
\end{eqnarray}
where $\|w\|_b = \sqrt{b(w,w)}$.
Then we define
\begin{eqnarray}
\delta_h(\lambda_i)=\sup_{w\in M(\lambda_i)}\inf_{v\in
V_{0h}(\Omega)}\|w-v\|_1.
\end{eqnarray}

For the eigenpair approximations by the finite element method, there
exist the following error estimates.
\begin{proposition}(\cite[Lemma 3.7, (3.28b,3.29b)]{Babuska2}, \cite[P. 699]{BabuskaOsborn} and
\cite{Chatelin})\label{Prop:Eigen_Error_Estimate}

\noindent(i) For any eigenfunction approximation $u_{i,h}$ of
(\ref{Discrete_Weak_Eigen_Problem}) $(i = 1, 2, \cdots, N_h)$, there is an
eigenfunction $u_i$ of (\ref{Weak_Eigenvalue_Problem}) corresponding to
$\lambda_i$ such that $\|u_i\|_b = 1$ and
\begin{eqnarray*}\label{Eigenfunction_Error}
\|u_i-u_{i,h}\|_{1,\Omega}&\leq& C_i\delta_h(\lambda_i).
\end{eqnarray*}
Furthermore,
\begin{eqnarray*}\label{Eigenfunction_Error_Negative}
\|u_i- u_{i,h}\|_{0,\Omega} &\leq& C_i\rho_{\Omega}(h)\delta_h(\lambda_i).
\end{eqnarray*}
(ii) For each eigenvalue, we have
\begin{eqnarray*}
\lambda_i \leq \lambda_{i,h}\leq \lambda_i + C_i\delta_h^2(\lambda_i).
\end{eqnarray*}
 Here and hereafter $C_i$ is some constant depending on $i$ but independent of  the mesh size $h$.
\end{proposition}
To analyze our method, we introduce the error expansion of
eigenvalue by the Rayleigh quotient formula which comes from
\cite{Babuska2,BabuskaOsborn,XuZhou_Eigen}.
\begin{proposition}\label{Prop:Rayleigh_Quotient_Error}
Assume $(\lambda,u)$ is the true solution of the eigenvalue problem
(\ref{Weak_Eigenvalue_Problem}) and  $0\neq \psi\in H_0^1(\Omega)$. Let us define
\begin{eqnarray*}
\widehat{\lambda}=\frac{a(\psi,\psi)}{b(\psi,\psi)}.
\end{eqnarray*}
Then we have
\begin{eqnarray*}
\widehat{\lambda}-\lambda
&=&\frac{a(u-\psi,u-\psi)}{b(\psi,\psi)}-\lambda
\frac{b(u-\psi,u-\psi)}{b(\psi,\psi)}.
\end{eqnarray*}
\end{proposition}

\section{Multilevel local and Parallel algorithms}\label{sec:LPA}
In this section, we present a new multilevel parallel algorithm to solve the eigenvalue problem
based on the combination of the local and parallel finite element technique and the multilevel correction method.
First, we introduce an one correction step with the local and parallel finite element scheme and then
present a parallel multilevel method for the eigevalue problem.

For the description of the numerical scheme, we need to define some notation.
Given an coarsest triangulation $\mathcal{T}_H(\Omega)$, we first divide the domain $\Omega$ into a number of disjoint
subdomains $D_1$, $\cdots$, $D_m$ such that $\bigcup_{j=1}^m\bar{D}_j=\bar{\Omega}$, $D_i\cap D_j=\emptyset$
(see Figure \ref{fig:DomainGi}), then enlarge each $D_j$ to obtain $\Omega_j$
that aligns with $\mathcal{T}_H(\Omega)$. We pick another sequence of subdomains
$G_j\subset\subset D_j\subset\Omega_j\subset \Omega$
 and (see Figure \ref{fig:DomainGi})
\begin{eqnarray*}
G_{m+1} = \Omega\setminus (\cup_{j=1}^m\bar{G_j}).
\end{eqnarray*}
\begin{figure}[htb]
\centering
\includegraphics[width=10cm,height=4.5cm]{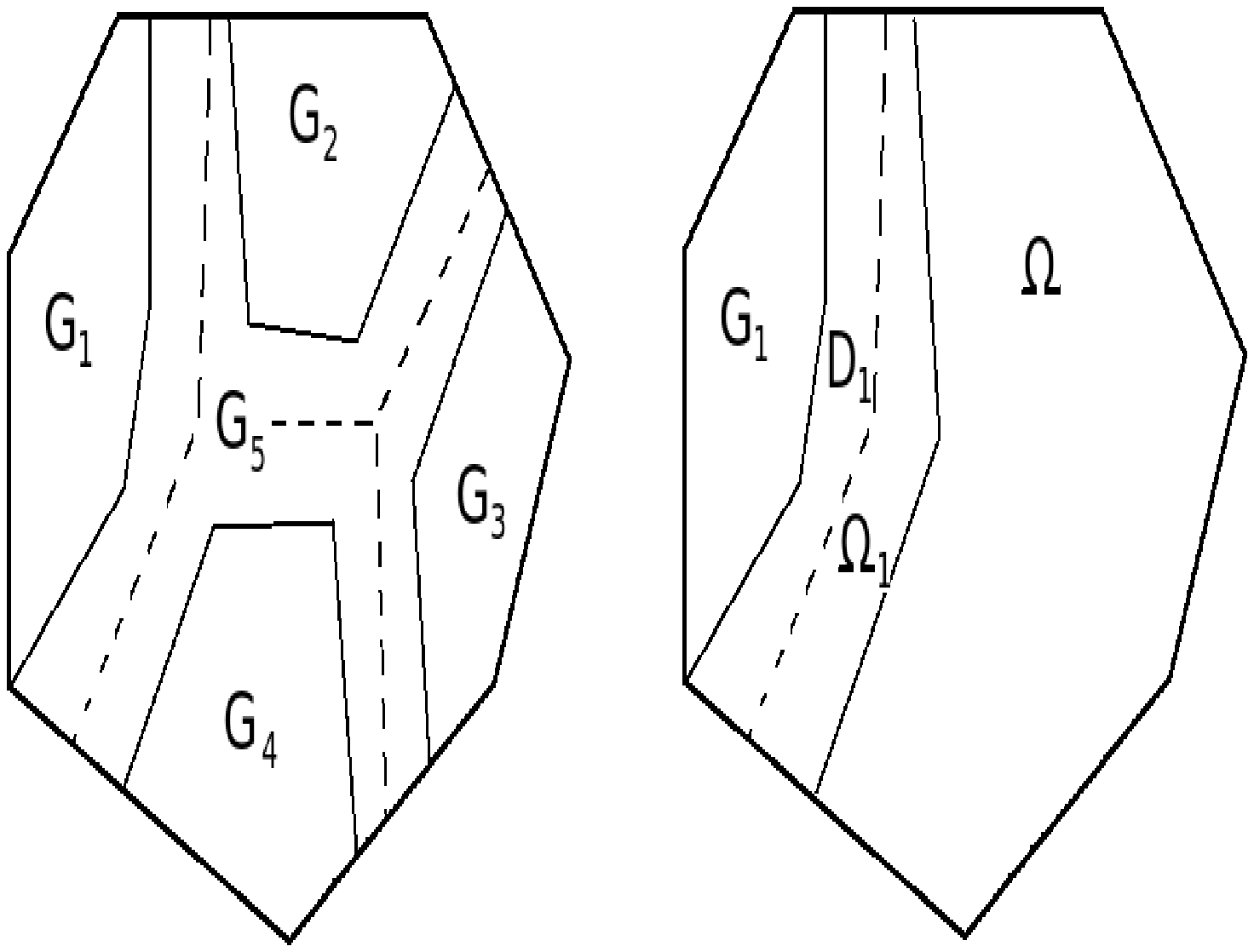}
\caption{$m=4$}\label{fig:DomainGi}
\end{figure}
In this paper we assume the domain decomposition satisfies the following property
\begin{equation}\label{Domain_Decomposition_Property}
\sum_{j=1}^{m}\|v\|_{\ell,\Omega_j}^2\lesssim \|v\|_{\ell,\Omega}^2
\end{equation}
for any $v\in H^{\ell}(\Omega)$ with $\ell=0,\ 1$.

\subsection{One correction step}
First, we present the one correction step  to improve the
accuracy of the given eigenvalue and eigenfunction approximation.
This correction method contains solving an auxiliary boundary value problem
in the finer finite element space on each subdomain
and an eigenvalue problem on the coarsest finite element space.

For simplicity of notation, we set
$(\lambda,u)=(\lambda_i,u_i)\ (i=1,2,\cdots,k,\cdots)$  and
$(\lambda_h, u_h)=(\lambda_{i,h},u_{i,h})\ (i=1,2,\cdots,N_h)$ to
denote an eigenpair and its corresponding approximation of problems (\ref{Weak_Eigenvalue_Problem}) and
(\ref{Discrete_Weak_Eigen_Problem}), respectively. For the clear understanding, we only describe the algorithm for
the simple eigenvalue case. The corresponding algorithm for the multiple eigenvalue case can be
given in the similar way as in \cite{Xie_Nonconforming}.

In order to do the correction step, we build original coarsest finite element space $V_{0H}(\Omega)$
on the background mesh $\mathcal{T}_H(\Omega)$. This coarsest finite element space $V_{0H}(\Omega)$ will be
used as the background space in our algorithm.

Assume we have obtained an eigenpair approximation
$(\lambda_{h_k},u_{h_k})\in\mathcal{R}\times V_{0h_k}(\Omega)$.
The one correction step will improve the accuracy of the
current eigenpair approximation $(\lambda_{h_k},u_{h_k})$.
Let $V_{0h_{k+1}}(\Omega)$ ba a finer finite element space such that
$V_{0h_k}(\Omega)\subset V_{0h_{k+1}}(\Omega)$. Here we assume the
finite element spaces $V_{0h_k}(\Omega)$ and $V_{0h_{k+1}}(\Omega)$
are consistent with the domain decomposition and $V_{0H}(\Omega)\subset V_{0h_k}(\Omega)$.
Based on this finer finite element space $V_{0h_{k+1}}(\Omega)$, we define the following one correction step.

\begin{algorithm}\label{Algm:One_Step_Correction}
One Correction Step

We have a given eigenpair approximation $(\lambda_{h_k},u_{h_k})\in\mathcal{R}\times V_{0h_k}(\Omega)$.
\begin{enumerate}
\item Define the following auxiliary boundary value problem:

For each $j = 1,2,\cdots,m$, find ${e}_{h_{k+1}}^j \in V_{0h_{k+1}}(\Omega_j)$ such that
\begin{equation}\label{Aux_Problem}
a({e}_{h_{k+1}}^j,v_{h_{k+1}})=\lambda_{h_k}b(u_{h_k},v_{h_{k+1}})-a({u}_{h_k},v_{h_{k+1}}),\ \
\ \forall v_{h_{k+1}}\in V_{0h_{k+1}}(\Omega_j).
\end{equation}
Set $\widetilde{u}_{h_{k+1}}^j=u_{h_k}+e_{h_{k+1}}^j\in V_{h_{k+1}}(\Omega_j)$.

\item  Construct $\widetilde{u}_{h_{k+1}}\in V_{0h_{k+1}}(\Omega)$ such that
$\widetilde{u}_{h_{k+1}} = \widetilde{u}_{h_{k+1}}^j$ in $G_j$ $(j = 1, \cdots, m)$
and $\widetilde{u}_{h_{k+1}} = \widetilde{u}_{h_{k+1}}^{m+1}$ in $G_{m+1}$  with
$\widetilde{u}_{h_{k+1}}^{m+1}$ being defined by solving the following problem:

Find $\widetilde{u}_{h_{k+1}}^{m+1}\in V_{h_{k+1}}(G_{m+1})$ such that
$\widetilde{u}_{h_{k+1}}^{m+1}|_{\partial G_j\cap \partial G_{m+1}}= \widetilde{u}_{h_{k+1}}^j$ $(j = 1, \cdots, m)$ and
\begin{equation}\label{G_m+1_Problem}
a(\widetilde{u}_{h_{k+1}}^{m+1}, v_{h_{k+1}}) = \lambda_{h_k} b(u_{h_k}, v_{h_{k+1}}),\ \ \ \forall v_{h_{k+1}}\in V_{0h_{k+1}}(G_{m+1}).
\end{equation}

\item  Define a new finite element
space $V_{H,h_{k+1}}=V_{0H}(\Omega)+{\rm span}\{\widetilde{u}_{h_{k+1}}\}$ and solve
the following eigenvalue problem:

Find $(\lambda_{h_{k+1}},u_{h_{k+1}})\in\mathcal{R}\times V_{H,h_{k+1}}$ such
that $b(u_{h_{k+1}},u_{h_{k+1}})=1$ and
\begin{eqnarray}\label{Eigen_Augment_Problem}
a(u_{h_{k+1}},v_{H,h_{k+1}})&=&\lambda_{h_{k+1}} b(u_{h_{k+1}},v_{H,h_{k+1}}),\ \ \
\forall v_{H,h_{k+1}}\in V_{H,h_{k+1}}.
\end{eqnarray}

\end{enumerate}

Summarize the above three steps into
\begin{eqnarray*}
(\lambda_{h_{k+1}},u_{h_{k+1}})={\it
Correction}(V_{0H}(\Omega),\lambda_{h_k}, u_{h_k},V_{0h_{k+1}}(\Omega)),
\end{eqnarray*}
where $\lambda_{h_k}$
and $u_{h_k}$ are the given eigenvalue and eigenfunction approximation, respectively.
\end{algorithm}

\begin{theorem}\label{Thm:Error_Estimate_One_Step_Correction}
Assume the current eigenpair approximation
$(\lambda_{h_k},u_{h_k})\in\mathcal{R}\times V_{0h_k}(\Omega)$ has the
following error estimates
\begin{eqnarray}
\|u-u_{h_k}\|_{1,\Omega} &\lesssim &\varepsilon_{h_k}(\lambda),\label{Estimate_u_u_h_k}\\
\|u-u_{h_k}\|_{0,\Omega}&\lesssim&\rho_{\Omega}(H)\varepsilon_{h_k}(\lambda),\label{Estimate_u_u_h_k_zero}\\
|\lambda-\lambda_{h_k}|&\lesssim&\varepsilon_{h_k}^2(\lambda).\label{Estimate_lambda_lambda_h_k}
\end{eqnarray}
Then after one step correction, the resultant approximation
$(\lambda_{h_{k+1}},u_{h_{k+1}})\in\mathcal{R}\times V_{0h_{k+1}}(\Omega)$ has the
following error estimates
\begin{eqnarray}
\|u-u_{h_{k+1}}\|_{1,\Omega} &\lesssim &\varepsilon_{h_{k+1}}(\lambda),\label{Estimate_u_u_h_{k+1}}\\
\|u-u_{h_{k+1}}\|_{0,\Omega}&\lesssim&
\rho_{\Omega}(H)\varepsilon_{h_{k+1}}(\lambda),\label{Estimate_u_u_h_{k+1}_zero}\\
|\lambda-\lambda_{h_{k+1}}|&\lesssim&\varepsilon_{h_{k+1}}^2(\lambda),\label{Estimate_lambda_lambda_h_{k+1}}
\end{eqnarray}
where
$\varepsilon_{h_{k+1}}(\lambda):=\rho_{\Omega}(H)\varepsilon_{h_k}(\lambda)
+\varepsilon_{h_k}^2(\lambda)+\delta_{h_{k+1}}(\lambda)$.
\end{theorem}

\begin{proof}
We focus on estimating $\|u-\widetilde{u}_{h_{k+1}}\|_{1,\Omega}$. First, we have
\begin{equation}\label{equ:Key0}
\|u-\widetilde{u}_{h_{k+1}}\|_{1,\Omega}
\lesssim \|u-P_{h_{k+1}}u\|_{1,\Omega} + \|\widetilde{u}_{h_{k+1}}-P_{h_{k+1}}u\|_{1,\Omega},
\end{equation}
and
\begin{equation}\label{equ:Key1}
\|\widetilde{u}_{h_{k+1}}-P_{h_{k+1}}u\|_{1,\Omega}^2
=\sum_{j=1}^m\|\widetilde{u}_{h_{k+1}}^j-P_{h_{k+1}}u\|_{1,G_j}^2
+\|\widetilde{u}_{h_{k+1}}^{m+1}-P_{h_{k+1}}u\|_{1,G_{m+1}}^2.
\end{equation}
From problems (\ref{Projection_Problem}), (\ref{Weak_Eigenvalue_Problem}) and
(\ref{Aux_Problem}), the following equation holds
\begin{equation*}
a(\widetilde{u}_{h_{k+1}}^j-P_{h_{k+1}}u,v)=b(\lambda_{h_k}u_{h_k}-\lambda u,v),\ \ \ \forall v\in V_{0h_{k+1}}(\Omega_{j}),
\end{equation*}
for $j=1,2,\cdots,m$.
According to Proposition \ref{Prop:Local_Estimate}
\begin{eqnarray}\label{equ:Key2}
&&\|\widetilde{u}_{h_{k+1}}^j-P_{h_{k+1}}u\|_{1,G_j} \lesssim
\|\widetilde{u}_{h_{k+1}}^j-P_{h_{k+1}}u\|_{0,\Omega_j}+\|\lambda_{h_k}u_{h_k}-\lambda u\|_{-1,\Omega_j}\nonumber\\
&\lesssim & \|\widetilde{u}_{h_{k+1}}^j-u_{h_k}\|_{0,\Omega_j}+\|u_{h_k}-P_{h_{k+1}}u\|_{0,\Omega_j}
+\|\lambda_{h_k}u_{h_k}-\lambda u\|_{0,\Omega_j}.
\end{eqnarray}
We will estimate the first term, i.e. $\|{e}_{h_{k+1}}^j\|_{0,\Omega_j}$ by
using the Aubin-Nitsche duality argument.

Given any $\phi\in L^2(\Omega_j)$, there exists $w^j\in H^1_0(\Omega_j)$ such that
\begin{eqnarray*}
a(v,w^j) &=& b(v,\phi),\ \ \ \forall v\in H^1_0(\Omega_j).
\end{eqnarray*}
Let $w^j_{h_{k+1}}\in V_{0h_{k+1}}(\Omega_j)$ and $w^j_{H}\in V_{0H}(\Omega_j)$ satisfying
\begin{eqnarray*}
a(v_{h_{k+1}},w^j_{h_{k+1}}) &=& a(v_{h_{k+1}},w^j),\ \ \ \forall v_{h_{k+1}}\in V_{0h_{k+1}}(\Omega_j),\\
a(v_{H},w^j_{H}) &=& a(v_{H},w^j),\ \ \ \ \  \ \forall v_{H}\in V_{0H}(\Omega_j).
\end{eqnarray*}
Then the following equations hold
\begin{eqnarray}\label{Equation_Nitsche}
&&b(\widetilde{u}_{h_{k+1}}^j-u_{h_k},\phi) = a(\widetilde{u}_{h_{k+1}}^j-u_{h_k},w^j_{h_{k+1}})\nonumber\\
&=& b(\lambda_{h_k}u_{h_k},w^j_{h_{k+1}})-a(u_{h_k},w^j_{h_{k+1}})\nonumber\\
&=& b(\lambda_{h_k}u_{h_k}-\lambda u,w^j_{h_{k+1}})+a(P_{h_{k+1}}u-u_{h_k},w^j_{h_{k+1}})\nonumber\\
&=& b(\lambda_{h_k}u_{h_k}-\lambda u,w^j_{h_{k+1}}-w^j_{H})+b(\lambda_{h_k}u_{h_k}-\lambda u,w^j_{H})\nonumber\\
&&\ \ \ \ \ \ +a(P_{h_{k+1}}{u}-u_{h_k},w^j_{h_{k+1}})\nonumber\\
&=& b(\lambda_{h_k}u_{h_k}-\lambda u,w^j_{h_{k+1}}-w^j_{H})+a(P_{h_{k+1}}{u}-u_{h_k},w^j_{h_{k+1}}-w^j_{H}),
\end{eqnarray}
where  $V_{0H}(\Omega)\subset V_{0h_k}(\Omega)$ and (\ref{Projection_Problem}), (\ref{Weak_Eigenvalue_Problem}), (\ref{Discrete_Weak_Eigen_Problem}), (\ref{Aux_Problem}) are used in the last equation.

Combining (\ref{Equation_Nitsche}) and the following error estimates
\begin{eqnarray*}
\|w-w^j_{h_{k+1}}\|_{1,\Omega_j} \lesssim \rho_{\Omega_j}(h_{k+1}) \|\phi\|_{0,\Omega_j},\ \ \
\|w-w^j_{H}\|_{1,\Omega_j} \lesssim \rho_{\Omega_j}(H) \|\phi\|_{0,\Omega_j},
\end{eqnarray*}
we have
\begin{equation}\label{equ:Key5}
\|\widetilde{u}_{h_{k+1}}^j-u_{h_k}\|_{0,\Omega_j}
\lesssim\rho_{\Omega_j}(H)\big(\|u_{h_k}-P_{h_{k+1}}u\|_{1,\Omega_j}+\|\lambda_{h_k}u_{h_k}-\lambda u\|_{0,\Omega_j}\big).
\end{equation}
From (\ref{equ:Key2}) and (\ref{equ:Key5}), for $j = 1,2\dots,m$, we have
\begin{eqnarray}\label{equ:Key3}
&&\|\widetilde{u}_{h_{k+1}}^j-P_{h_{k+1}}u\|_{1,G_j}\lesssim \rho_{\Omega_j}(H)\|u_{h_k}-P_{h_{k+1}}u\|_{1,\Omega_j}\nonumber\\
&&\ \ \ \ \ \ \ \ \ \ \ \ \
\ \ \ \ \  +\|u_{h_k}-P_{h_{k+1}}u\|_{0,\Omega_j}+\|\lambda_{h_k}u_{h_k}-\lambda u\|_{0,\Omega_j}.
\end{eqnarray}

Now, we estimate $\|\widetilde{u}_{h_{k+1}}^{m+1}-P_{h_{k+1}}u\|_{1,G_{m+1}}$.
From (\ref{Projection_Problem}), (\ref{Weak_Eigenvalue_Problem}) and (\ref{G_m+1_Problem}), we obtain
\begin{equation*}
a(\widetilde{u}_{h_{k+1}}^{m+1}-P_{h_{k+1}}u,v)=b(\lambda_{h_k}u_{h_k}-\lambda u,v),\ \ \ \forall v\in V_{0h_{k+1}}(G_{m+1}).
\end{equation*}
For any $v\in V_{0h_{k+1}}(G_{m+1})$, the following estimates hold
\begin{eqnarray}\label{Estimate_G_m+1}
&&\|\widetilde{u}_{h_{k+1}}^{m+1}-P_{h_{k+1}}u\|_{1,G_{m+1}}^2\nonumber\\
&\lesssim& a(\widetilde{u}_{h_{k+1}}^{m+1}-P_{h_{k+1}}u,\widetilde{u}_{h_{k+1}}^{m+1}-P_{h_{k+1}}u-v)
+b(\lambda_{h_k}u_{h_k}-\lambda u,v)\nonumber\\
&\lesssim & \|\widetilde{u}_{h_{k+1}}^{m+1}-P_{h_{k+1}}u\|_{1,G_{m+1}}
\inf_{\chi\in V_{0h_{k+1}}(G_{m+1})}\|\widetilde{u}_{h_{k+1}}^{m+1}-P_{h_{k+1}}u-\chi\|_{1,G_{m+1}}\nonumber\\
&&+\|\lambda_{h_k} u_{h_k}-\lambda u\|_{-1,G_{m+1}}
\big(\|\widetilde{u}_{h_{k+1}}^{m+1}-P_{h_{k+1}}u\|_{1,G_{m+1}}\nonumber\\
&& \ \ \ \ \ \
+\inf_{\chi\in V_{0h_{k+1}}(G_{m+1})}\|\widetilde{u}_{h_{k+1}}^{m+1}-P_{h_{k+1}}u-\chi\|_{1,G_{m+1}}\big).
\end{eqnarray}
Combining (\ref{Estimate_G_m+1}) and the following estimate
\begin{eqnarray*}
\|\widetilde{u}_{h_{k+1}}^{m+1}-P_{h_{k+1}}u\|_{1/2,\partial G_{m+1}}^2
&\lesssim& \sum_{j=1}^m\|\widetilde{u}_{h_{k+1}}^j-P_{h_{k+1}}u\|_{1/2,\partial G_{j}}^2\nonumber\\
&\lesssim& \sum_{j=1}^m\|\widetilde{u}_{h_{k+1}}^j-P_{h_{k+1}}u\|_{1,G_{j}}^2
\end{eqnarray*}
and Proposition \ref{Proposition_Fractional_Norm}, we have
\begin{eqnarray}\label{equ:Key4}
&&\|\widetilde{u}_{h_{k+1}}^{m+1}-P_{h_{k+1}}u\|_{1,G_{m+1}}^2\nonumber\\
&\lesssim & \inf_{\chi\in V_{0h_{k+1}}(G_{m+1})}\|\widetilde{u}_{h_{k+1}}^{m+1}-P_{h_{k+1}}u-\chi\|_{1,G_{m+1}}^2
+\|\lambda_{h_k} u_{h_k}-\lambda u\|_{-1,G_{m+1}}^2\nonumber\\
&\lesssim & \|\widetilde{u}_{h_{k+1}}-P_{h_{k+1}}u\|_{1/2,\partial G_{m+1}}^2
+\|\lambda_{h_k} u_{h_k}-\lambda u\|_{0,G_{m+1}}^2\nonumber\\
&\lesssim &
\sum_{j=1}^m\|\widetilde{u}_{h_{k+1}}^j-P_{h_{k+1}}u\|_{1,G_{j}}^2
+\|\lambda_{h_k} u_{h_k}-\lambda u\|_{0,G_{m+1}}^2.
\end{eqnarray}
Combining (\ref{Domain_Decomposition_Property}), (\ref{equ:Key1}), (\ref{equ:Key3})
and (\ref{equ:Key4}) leads to
\begin{eqnarray*}
&&\|\widetilde{u}_{h_{k+1}}-P_{h_{k+1}}u\|_{1,\Omega}^2\nonumber\\
&\lesssim& \sum_{j=1}^m\rho_{\Omega_j}(H)^2\|u_{h_k}-P_{h_{k+1}}u\|_{1,\Omega_j}
+\sum_{j=1}^m\|{u}_{h_k}-P_{h_{k+1}}u\|_{0,\Omega_{j}}^2\nonumber\\
&&\ \ \ +\sum_{j=1}^m\|\lambda_{h_k} u_{h_k}-\lambda u\|_{0,\Omega_{j}}^2+\|\lambda_{h_k} u_{h_k}-\lambda u\|_{0,G_{m+1}}^2\nonumber\\
&\lesssim& \rho^2_{\Omega}(H)\|{u}_{h_k}-P_{h_{k+1}}u\|_{1,\Omega}^2+\|{u}_{h_k}-P_{h_{k+1}}u\|_{0,\Omega}^2
+\|\lambda_{h_k} u_{h_k}-\lambda u\|_{0,\Omega}^2\nonumber\\
&\lesssim&\rho^2_{\Omega}(H)\|{u}_{h_k}-u\|_{1,\Omega}^2
+\rho^2_{\Omega}(H)\|u-P_{h_{k+1}}u\|_{1,\Omega}^2+\|{u}_{h_k}-u\|_{0,\Omega}^2\nonumber\\
&&+\|u-P_{h_{k+1}}u\|_{0,\Omega}^2+|\lambda-\lambda_{h_k}|^2\|u\|_{0,\Omega}^2
+\lambda^2\|u_{h_k}- u\|_{0,\Omega}^2.
\end{eqnarray*}
Together with the error estimate of the finite element projection
\begin{eqnarray*}
\|u-P_{h_{k+1}}u\|_{1,\Omega} &\lesssim&\delta_{h_{k+1}}(\lambda)
\end{eqnarray*}
and (\ref{Estimate_lambda_lambda_h_k}), (\ref{equ:Key0}), we have
\begin{eqnarray}\label{Error_tilde_u_h_{k+1}}
\|u-\widetilde{u}_{h_{k+1}}\|_{1,\Omega}&\lesssim&\|u-P_{h_{k+1}}u\|_{1,\Omega}+
|\lambda-\lambda_{h_k}|+\|u-u_{h_k}\|_{0,\Omega}\nonumber\\
&&\ \ \ +\rho_{\Omega}(H)\|u-u_{h_k}\|_{1,\Omega}\nonumber\\
&\lesssim&\rho_{\Omega}(H)\varepsilon_{h_k}(\lambda)+\varepsilon_{h_k}^2(\lambda)
+\delta_{h_{k+1}}(\lambda).
\end{eqnarray}
From (\ref{Error_u_h_{k+1}}) and (\ref{Error_tilde_u_h_{k+1}}), we can obtain (\ref{Estimate_u_u_h_{k+1}}).

We come to estimate the error for the eigenpair solution
$(\lambda_{h_{k+1}},u_{h_{k+1}})$ of problem (\ref{Eigen_Augment_Problem}).
Based on the error estimate theory of eigenvalue problems by finite
element methods (see, e.g., Proposition \ref{Prop:Eigen_Error_Estimate} or
\cite[Theorem 9.1]{BabuskaOsborn}) and the definition of the space $V_{H,h_{k+1}}$,
the following estimates hold
\begin{eqnarray}\label{Error_u_h_{k+1}}
\|u-u_{h_{k+1}}\|_{1,\Omega}&\lesssim& \sup_{w\in M(\lambda)}\inf_{v\in
V_{H,h_{k+1}}}\|w-v\|_{1,\Omega}\lesssim \|u-\widetilde{u}_{h_{k+1}}\|_{1,\Omega},
\end{eqnarray}
and
\begin{eqnarray*}\label{Error_u_h_{k+1}_negative}
\|u-u_{h_{k+1}}\|_{0,\Omega}&\lesssim&\widetilde{\rho}_{\Omega}(H)\|u-u_{h_{k+1}}\|_{1,\Omega},
\end{eqnarray*}
where
\begin{eqnarray*}\label{Eta_a_H}
\widetilde{\rho}_{\Omega}(H)&=&\sup_{f\in V,\|f\|_{0,\Omega}=1}\inf_{v\in
V_{H,h_{k+1}}}\|L^{-1}f-v\|_{1,\Omega} \leq \rho_{\Omega}(H).
\end{eqnarray*}
So we obtain the desired result (\ref{Estimate_u_u_h_{k+1}}), (\ref{Estimate_u_u_h_{k+1}_zero}) and
%\begin{eqnarray*}
%\|u-u_{h_{k+1}}\|_{0,\Omega}&\lesssim&
%\rho_{\Omega}(H)\|u-u_{h_{k+1}}\|_{1,\Omega}.
%\end{eqnarray*}
the estimate (\ref{Estimate_lambda_lambda_h_{k+1}}) can be obtained by
Proposition \ref{Prop:Rayleigh_Quotient_Error} and (\ref{Estimate_u_u_h_{k+1}}).
\end{proof}

\subsection{Multilevel correction process}
Now we introduce a type of multilevel local and parallel
scheme based on the one correction step defined in Algorithm
\ref{Algm:One_Step_Correction}. This type of multilevel method can obtain the same optimal error
estimate as solving the eigenvalue problem directly in the finest
finite element space.

In order to do multilevel local and parallel
scheme, we define a sequence of triangulations $\mathcal{T}_{h_k}(\Omega)$ of $\Omega$ determined as follows.
Suppose $\mathcal{T}_{h_1}(\Omega)$ is obtained from $\mathcal{T}_H(\Omega)$ by the regular refinement
 and let $\mathcal{T}_{h_k}(\Omega)$ be obtained from $\mathcal{T}_{h_{k-1}}(\Omega)$ via
 regular refinement (produce $\beta^d$ congruent elements) such that
$$h_k\approx\frac{1}{\beta}h_{k-1}\ \ \ \ \ {\rm for}\ k\geq 2.$$
Based on this sequence of meshes, we construct the corresponding linear finite element spaces such that
for each $j = 1,2,\cdots,m$
\begin{eqnarray*}
V_{0H}(\Omega_j)\subset V_{0h_1}(\Omega_j)\subset V_{0h_2}(\Omega_j)\subset\cdots\subset V_{0h_n}(\Omega_j)
\end{eqnarray*}
and the following relation of approximation errors holds
\begin{eqnarray}\label{Error_k_k_1}
\delta_{h_k}(\lambda)\approx\frac{1}{\beta}\delta_{h_{k-1}}(\lambda),\ \ \ k=2,\cdots,n.
\end{eqnarray}

\begin{remark}
The relation (\ref{Error_k_k_1}) is reasonable since we can choose
$\delta_{h_k}(\lambda)=h_k\ (k=1,\cdots,n)$. Always the upper bound of
the estimate $\delta_{h_k}(\lambda)\lesssim h_k$ holds. Recently, we also obtain the
lower bound $\delta_{h_k}(\lambda)\gtrsim h_k$ (c.f. \cite{LinXieXu}).
\end{remark}

\begin{algorithm}\label{Algm:Multi_Correction}
Multilevel Correction Scheme
\begin{enumerate}
\item Solve the following eigenvalue problem in $V_{0h_1}(\Omega)$:

Find $(\lambda_{h_1},u_{h_1})\in \mathcal{R}\times V_{0h_1}(\Omega)$ such that
$b(u_{h_1},u_{h_1})=1$ and
\begin{equation*}%\label{Initial_Eigen_Problem}
a(u_{h_1},v_{h_1})=\lambda_{h_1}b(u_{h_1},v_{h_1}),\ \ \ \ \forall v_{h_1}\in V_{0h_1}(\Omega).
\end{equation*}
\item Construct a series of finer finite element
spaces $V_{0h_2}(\Omega_j),\cdots,V_{0h_n}(\Omega_j)$
such that $\rho_{\Omega}(H)\gtrsim
\delta_{h_1}(\lambda)\geq \delta_{h_2}(\lambda)\geq\cdots\geq
\delta_{h_n}(\lambda)$ and (\ref{Error_k_k_1}) holds.
\item Do $k=1,\cdots,n-1$
\begin{itemize}
\item Obtain a new eigenpair approximation
$(\lambda_{h_{k+1}},u_{h_{k+1}})\in \mathcal{R}\times V_{0h_{k+1}}(\Omega)$
by Algorithm \ref{Algm:One_Step_Correction}
\begin{eqnarray*}
(\lambda_{h_{k+1}},u_{h_{k+1}})={\it Correction}(V_{0H}(\Omega),\lambda_{h_k},u_{h_k},V_{0h_{k+1}}(\Omega)).
\end{eqnarray*}
\end{itemize}
end Do
\end{enumerate}
Finally, we obtain an eigenpair approximation $(\lambda_{h_n},u_{h_n})\in\mathcal{R}\times V_{0h_n}(\Omega)$.
\end{algorithm}

\begin{theorem}\label{Thm:Multi_Correction}
After implementing Algorithm \ref{Algm:Multi_Correction}, there exists an eigenfunction
$u\in M(\lambda)$  such that the resultant
eigenpair approximation $(\lambda_{h_n},u_{h_n})$ has the following
error estimate
\begin{eqnarray}
\|u-u_{h_n}\|_{1,\Omega} &\lesssim&\delta_{h_n}(\lambda),\label{Multi_Correction_Err_fun1}\\
\|u-u_{h_n}\|_{0,\Omega}&\lesssim&\rho_{\Omega}(H)\delta_{h_n}(\lambda),\label{Multi_Correction_Err_fun0}\\
|\lambda-\lambda_{h_n}|&\lesssim&\delta_{h_n}^2(\lambda),\label{Multi_Correction_Err_eigen}
\end{eqnarray}
under the condition $C\beta\rho_{\Omega}(H)<1$ for some constant $C$.
\end{theorem}
%--------------------------------------------------------------------------------------------------------
\begin{proof}
Based on Proposition \ref{Prop:Eigen_Error_Estimate}, there exists an eigenfunction $u\in M(\lambda)$ such that
\begin{eqnarray}
|\lambda-\lambda_{h_1}| &\lesssim & \delta_{h_1}^2(\lambda),\label{Initial_Error_Eigenvalue}\\
\|u-u_{h_1}\|_{1,\Omega} &\lesssim& \delta_{h_1}(\lambda),\label{Initial_Error_Eigenfunc_1}\\
\|u-u_{h_1}\|_{0,\Omega}&\lesssim& \rho_{\Omega}(h_1)\delta_{h_1}(\lambda).\label{Initial_Error_Eigenfunc_0}
\end{eqnarray}
Let $\varepsilon_{h_1}(\lambda):=\delta_{h_1}(\lambda)$. From
(\ref{Initial_Error_Eigenvalue})-(\ref{Initial_Error_Eigenfunc_0}) and
Theorem \ref{Thm:Error_Estimate_One_Step_Correction}, we have
\begin{eqnarray*}
\varepsilon_{h_{k+1}}(\lambda)
&\lesssim&\rho_{\Omega}(H)\varepsilon_{h_k}(\lambda)
+\varepsilon_{h_k}^2(\lambda)+\delta_{h_{k+1}}(\lambda)\\
&\lesssim&\rho_{\Omega}(H)\varepsilon_{h_k}(\lambda)+\delta_{h_{k+1}}(\lambda),
\ \ \ \ {\rm for}\ 1\leq k\leq n-1.
\end{eqnarray*}
by a process of induction with the condition $\rho_{\Omega}(H)\gtrsim\delta_{h_1}(\lambda)\geq
\delta_{h_2}(\lambda)\geq\cdots\geq \delta_{h_n}(\lambda)$.
Then by recursive relation, we obtain
\begin{eqnarray}\label{varepsilon_h_n}
\varepsilon_{h_{n}}(\lambda)
&\lesssim&\rho_{\Omega}(H)\varepsilon_{h_{n-1}}(\lambda)+\delta_{h_{n}}(\lambda)\nonumber\\
&\lesssim&\rho^2_{\Omega}(H)\varepsilon_{h_{n-2}}(\lambda)+
\rho_{\Omega}(H)\delta_{h_{n-1}}(\lambda)+\delta_{h_{n}}(\lambda)\nonumber\\
&\lesssim&\sum\limits_{k=1}^{n}(\rho_{\Omega}(H))^{n-k}\delta_{h_k}(\lambda).
\end{eqnarray}
Based on the proof in Theorem \ref{Thm:Error_Estimate_One_Step_Correction}, (\ref{Error_k_k_1})
and (\ref{varepsilon_h_n}), the final eigenfunction approximation $u_{h_n}$ has the error estimate
\begin{eqnarray*}\label{Error_u_h_n_Multi_Correction}
\|u-u_{h_n}\|_{1,\Omega}&\lesssim&\varepsilon_{h_{n}}(\lambda)
\lesssim \sum_{k=1}^n(\rho_{\Omega}(H))^{n-k}\delta_{h_k}(\lambda)\\
&=&\sum_{k=1}^n\big(\beta\rho_{\Omega}(H)\big)^{n-k}\delta_{h_n}(\lambda)
\lesssim \frac{\delta_{h_n}(\lambda)}{1-\beta\rho_{\Omega}(H)}\\
&\lesssim&\delta_{h_n}(\lambda).
\end{eqnarray*}
The desired result (\ref{Multi_Correction_Err_fun0}) and
(\ref{Multi_Correction_Err_eigen}) can also be proved with the similar way in the proof
of Theorem \ref{Thm:Error_Estimate_One_Step_Correction}.
\end{proof}

\section{Work estimate of algorithm}
In this section, we turn our attention to the estimate of computational work
for Algorithm \ref{Algm:Multi_Correction}. We will show that
Algorithm \ref{Algm:Multi_Correction} makes solving eigenvalue problem need almost the
same work as solving the boundary value problem by the local and parallel finite element method.

First, we define the dimension of each level linear
finite element space as
\begin{equation*}
N_k^j:={\rm dim}V_{0h_k}(\Omega_j)\text{ and }N_k:={\rm dim}V_{0h_k}(\Omega),\ \
k=1,\cdots,n,\ j=1,\cdots,m+1.
\end{equation*}
Then we have
\begin{equation}\label{relation_dimension}
N_k^j \thickapprox\Big(\frac{1}{\beta}\Big)^{d(n-k)}N_n^j\ \ {\rm and}\ \
N_k^j\approx \frac{N_k}{m},\ \ \ k=1,\cdots, n.
\end{equation}
\begin{theorem}
Assume the eigenvalue problem solving in the coarsest spaces $V_{0H}(\Omega)$ and $V_{0h_1}(\Omega)$ need work
$\mathcal{O}(M_H)$ and $\mathcal{O}(M_{h_1})$, respectively,
and the work of solving the boundary value problem in $V_{h_k}(\Omega_j)$ and $V_{h_k}(G_{m+1})$ be
$\mathcal{O}(N_k^j)$ and $\mathcal{O}(N_k^{m+1})$,
$\forall k=1,2,\cdots,n \text{ and } j = 1,2,\cdots,m$.
Then the work involved in Algorithm \ref{Algm:Multi_Correction} is
 $\mathcal{O}(N_n/m+M_H\log N_n+M_{h_1})$ for each processor.
Furthermore, the complexity in each processor
will be $\mathcal{O}(N_n/m)$ provided $M_H\ll N_n/m$ and $M_{h_1}\leq N_n/m$.
\end{theorem}
\begin{proof}
Let $W_k$ denote the work in any processor of the one correction step in
the $k$-th finite element space $V_{h_k}$. Then with the definition, we have
\begin{eqnarray}\label{work_k}
W_k&=&\mathcal{O}(N_k/m+M_H)\ \ \ \ {\rm for}\ k\geq 2.
\end{eqnarray}
Iterating (\ref{work_k}) and using the fact (\ref{relation_dimension}), we obtain
\begin{eqnarray}\label{Work_Estimate}
&&\text{The total work in any processor}\leq\sum_{k=1}^nW_k\nonumber\\
&=& \mathcal{O}\Big(M_{h_1}+\sum_{k=2}^n\big(N_k/m+M_H\big)\Big)\nonumber\\
&=&\mathcal{O}\Big(\sum_{k=2}^nN_k/m+(n-2)M_H+M_{h_1}\Big)\nonumber\\
&=&\mathcal{O}\Big(\sum_{k=2}^n\big(\frac{1}{\beta}\big)^{d(n-k)}N_n/m+(n-2)M_H+M_{h_1}\Big)\nonumber\\
&=&\mathcal{O}(N_n/m+M_H\log N_n+M_{h_1}).
\end{eqnarray}
This is the desired result $\mathcal{O}(N_n/m+M_H\log N_n+M_{h_1})$ and the
one $\mathcal{O}(N_n/m)$  can be obtained by the conditions $M_H\ll N_n/m$ and $M_{h_1}\leq N_n/m$.
\end{proof}
\begin{remark}
The linear complexity $\mathcal{O}(N_k^j)$ and $\mathcal{O}(N_k^{m+1})$
can be arrived by the so-called multigrid method (see, e.g., \cite{Bramble,BrambleZhang,Hackbush,McCormick,Xu}).
\end{remark}

\section{Numerical result}\label{Numerical_Result_Section}
In this section, we give two numerical examples to illustrate the
efficiency of the multilevel correction algorithm {(Algorithm \ref{Algm:Multi_Correction})}
proposed in this paper.

\begin{example}\label{Example_1}
In this example, the eigenvalue problem (\ref{Weak_Eigenvalue_Problem}) is solved on the square  $\Omega=(-1,1)\times(-1,1)$ with $a(u,v) = \int_{\Omega}\nabla u\cdot\nabla v\mathrm{d}\Omega$
and ${b(u,v)} = \int_{\Omega}u v \mathrm{d}\Omega$.
\end{example}
As {in} Figure \ref{fig:ComputeDomain}, we first divide the domain $\Omega$ into
{ four disjoint subdomains} $D_1$, $\cdots$, $D_4$ such that
$\bigcup_{j=1}^4\bar{D}_j=\bar{\Omega}$, $D_i\cap D_j=\emptyset$, {
then enlarge each $D_j$ to obtain $\Omega_j$ such that
 $G_j\subset\subset D_j \subset\Omega_j\subset \Omega$ for $i,j=1,2,3,4$} and
\begin{equation*}
G_{5} = \Omega\setminus (\cup_{j=1}^4\bar{G_j}).
\end{equation*}
\begin{figure}[htb]
\centering\includegraphics[width=8cm,height=4cm]{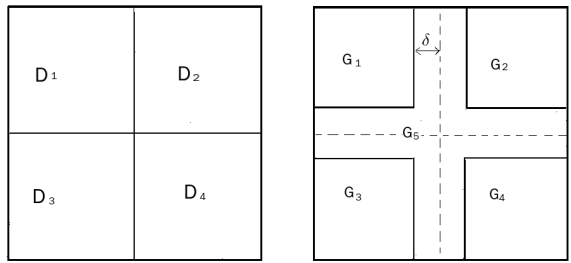}
\caption{$\bigcup_{j=1}^4\bar{D}_j=\bar{\Omega}$, $G_{5} = \Omega\setminus (\cup_{j=1}^4\bar{G_j})$ }
\label{fig:ComputeDomain}
\end{figure}
The sequence of finite element spaces is constructed by
using the linear or quadratic element on the nested sequence of triangulations which are produced by the
regular refinement with $\beta =2$ (connecting the midpoints of each edge).

Algorithm \ref{Algm:Multi_Correction} is applied to solve the eigenvalue problem. If the linear element is used,
from Theorem \ref{Thm:Multi_Correction}, we have the following error estimates for eigenpair approximation
\begin{eqnarray*}
|\lambda_{h_n}-\lambda| \lesssim h_n^2,\ \ \|u_{h_n}-u\|_{1,\Omega} \lesssim h_n
\end{eqnarray*}
which means the multilevel correction method can also obtain the optimal convergence order.

The numerical results for the first five eigenvalues and the $1$-st, $4$-th eigenfunctions (they are simple)
{by the linear finite element method with five levels grids are shown in
 Tables \ref{tab:num_res_eva} and \ref{tab:num_res_eve}.}
  It is observed from Tables \ref{tab:num_res_eva} and \ref{tab:num_res_eve} that the
numerical results confirm the efficiency of the proposed algorithm.
\begin{table}[]
\centering
\caption{The errors for the first $5$ eigenvalue approximations}
\label{tab:num_res_eva}
\begin{tabular}{c|ccccc}
\hline
Eigenvalues & $|\lambda-\lambda_{h_1}|$ & $|\lambda-\lambda_{h_2}|$ & $|\lambda-\lambda_{h_3}|$ & $|\lambda-\lambda_{h_4}|$ & $|\lambda-\lambda_{h_5}|$\\
\hline
1-st     & 0.073555 & 0.018534 & 0.004651 & 0.001164 & 0.000291 \\
Order & -- & 1.988649 & 1.994561 & 1.998450 & 2.000000     \\
\hline
2-nd     & 0.426525 & 0.106936 & 0.026747 & 0.006689 & 0.001673 \\
Order & -- & 1.995883 & 1.999299 & 1.999515 & 1.999353 \\
\hline
3-rd     & 0.426534 & 0.106939 & 0.026748 & 0.006689 & 0.001673 \\
Order & -- & 1.995873 & 1.999285 & 1.999569 & 1.999353 \\
\hline
4-th     & 1.078632 & 0.267624 & 0.066859 & 0.016717 & 0.004180\\
Order & -- & 2.010923 & 2.001014 & 1.999806 & 1.999741 \\
\hline
5-th     & 1.490468 & 0.385000 & 0.097106 & 0.024349 & 0.006093 \\
Order & -- & 1.952835 & 1.987226 & 1.995698 & 1.998638 \\
\hline
\end{tabular}
\end{table}

\begin{table}[]
\centering
\caption{The errors for the simple ($1$-st and $5$-th) eigenfunction approximations}\label{tab:num_res_eve}
\begin{tabular}{c|ccccc}
\hline
{\footnotesize Eigenfunctions} & {\footnotesize$\|u-u_{h_1}\|_{1,\Omega}$} & {\footnotesize$\|u-u_{h_2}\|_{1,\Omega}$} & {\footnotesize$\|u-u_{h_3}\|_{1,\Omega}$} &
{\footnotesize$\|u-u_{h_4}\|_{1,\Omega}$} & {\footnotesize$\|u-u_{h_5}\|_{1,\Omega}$}\\
\hline
1-st     & 0.269991 & 0.135956 & 0.068195 & 0.034119 & 0.017064 \\
Order & -- & 0.989771 & 0.995402 & 0.999091 & 0.999619 \\
\hline
4-th     & 1.025704 & 0.514925 & 0.259424 & 0.129254 & 0.064645 \\
Order & -- & 0.994180 & 0.989050 & 1.005103 & 0.999598 \\
\hline
\end{tabular}
\end{table}

Next we discuss the effectiveness of $\delta$ and the coarsest mesh size $H$ to
the numerical results by Algorithm \ref{Algm:Multi_Correction}. Figure \ref{fig:P1HD}
shows the errors for the different choices of $\delta$ and $H$ by the linear finite element method.
From Figure \ref{fig:P1HD}, we can find Algorithm \ref{Algm:Multi_Correction} can obtain the optimal
convergence order when $H\leq 0.25$ and $\delta\geq 0.1$ which are soft requirements.
\begin{figure}[htb]
\centering
\includegraphics[width=5cm,height=5cm]{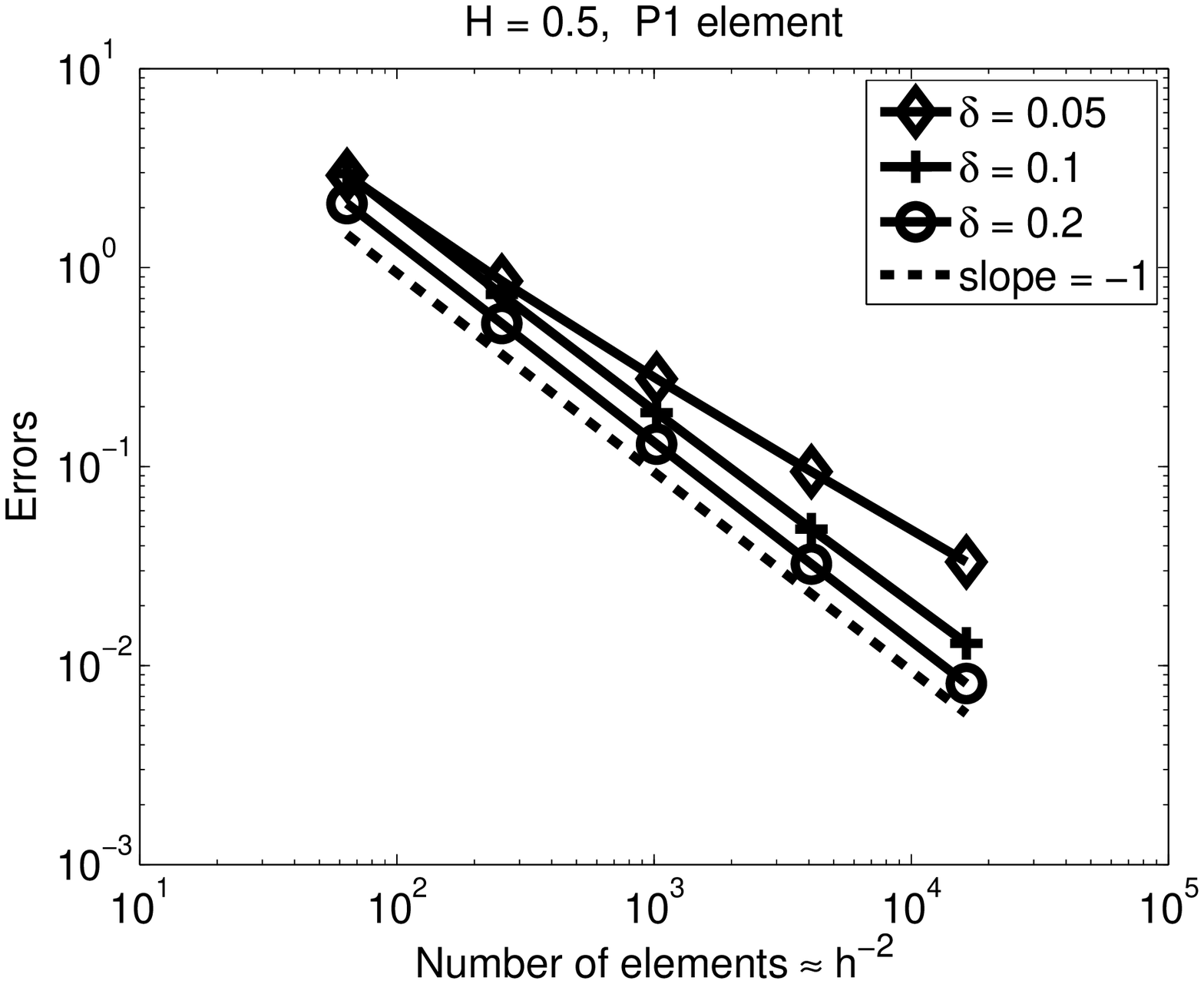}
\includegraphics[width=5cm,height=5cm]{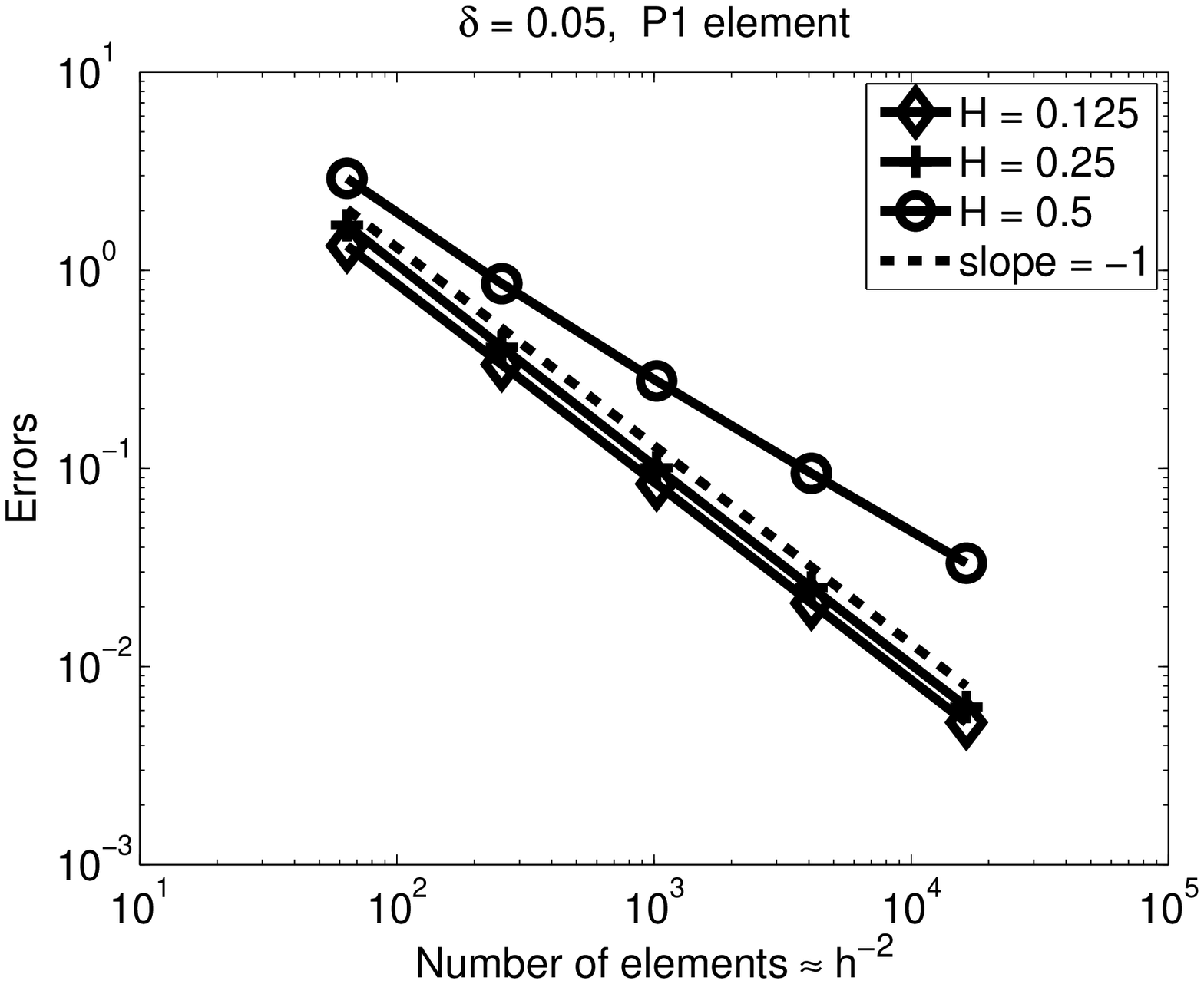}
\caption{The error estimate for the first $5$ eigenvalue approximations by the linear
element: The left subfigure is for $H=0.5$ and $\delta= 0.05$, $0.1$, $0.2$. The right
 subfigure is for  $\delta=0.05$ and $H=0.5$, $0.25$, $0.125$}\label{fig:P1HD}
\end{figure}
The case becomes better when we use the quadratic finite element method (see Figure \ref{fig:P2HD}).
For the quadratic finite element, the convergence order ($4$-th) is always optimal
even when $\delta$ is very small.
\begin{figure}[htb]
\centering
\includegraphics[width=5cm,height=5cm]{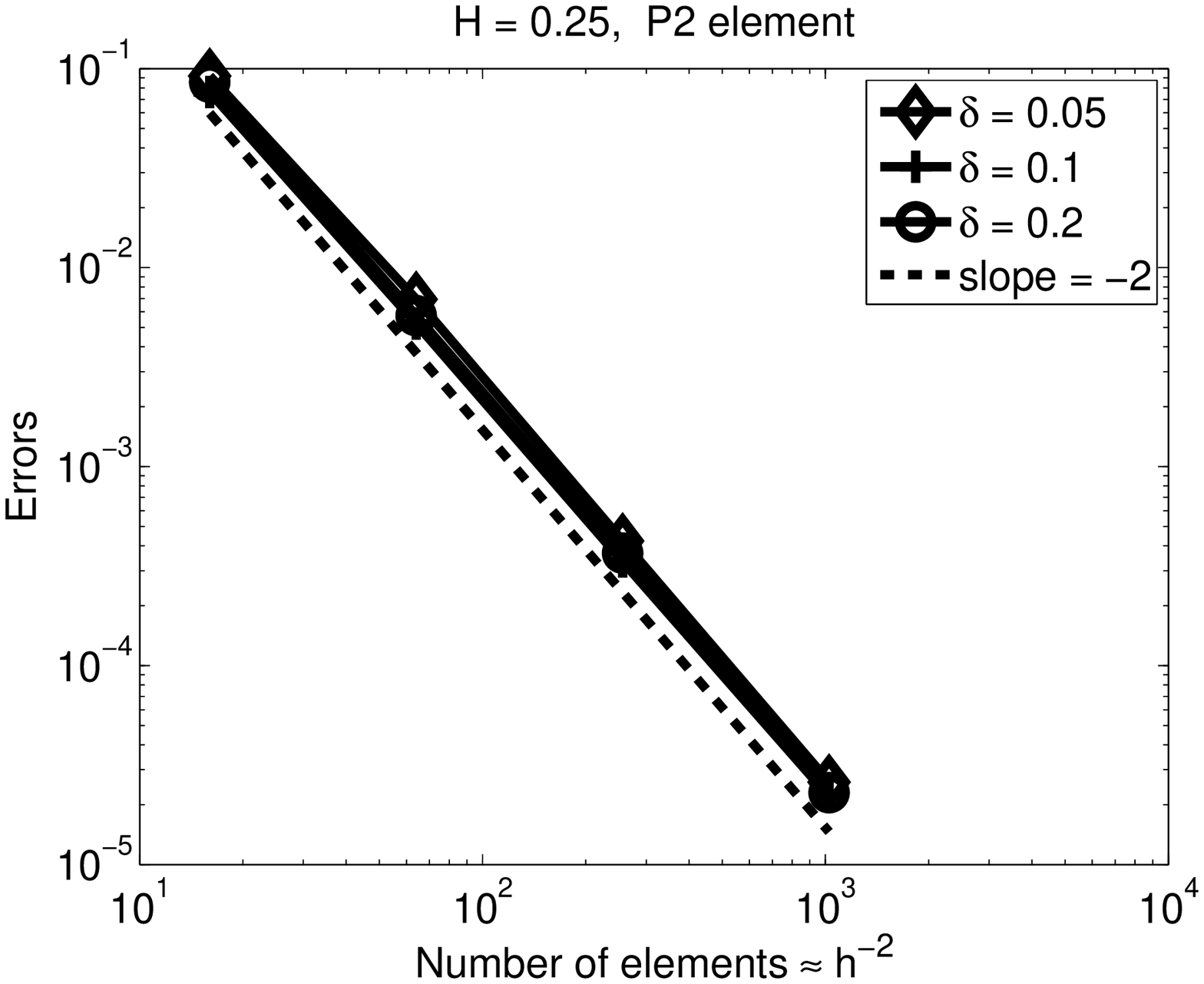}
\includegraphics[width=5cm,height=5cm]{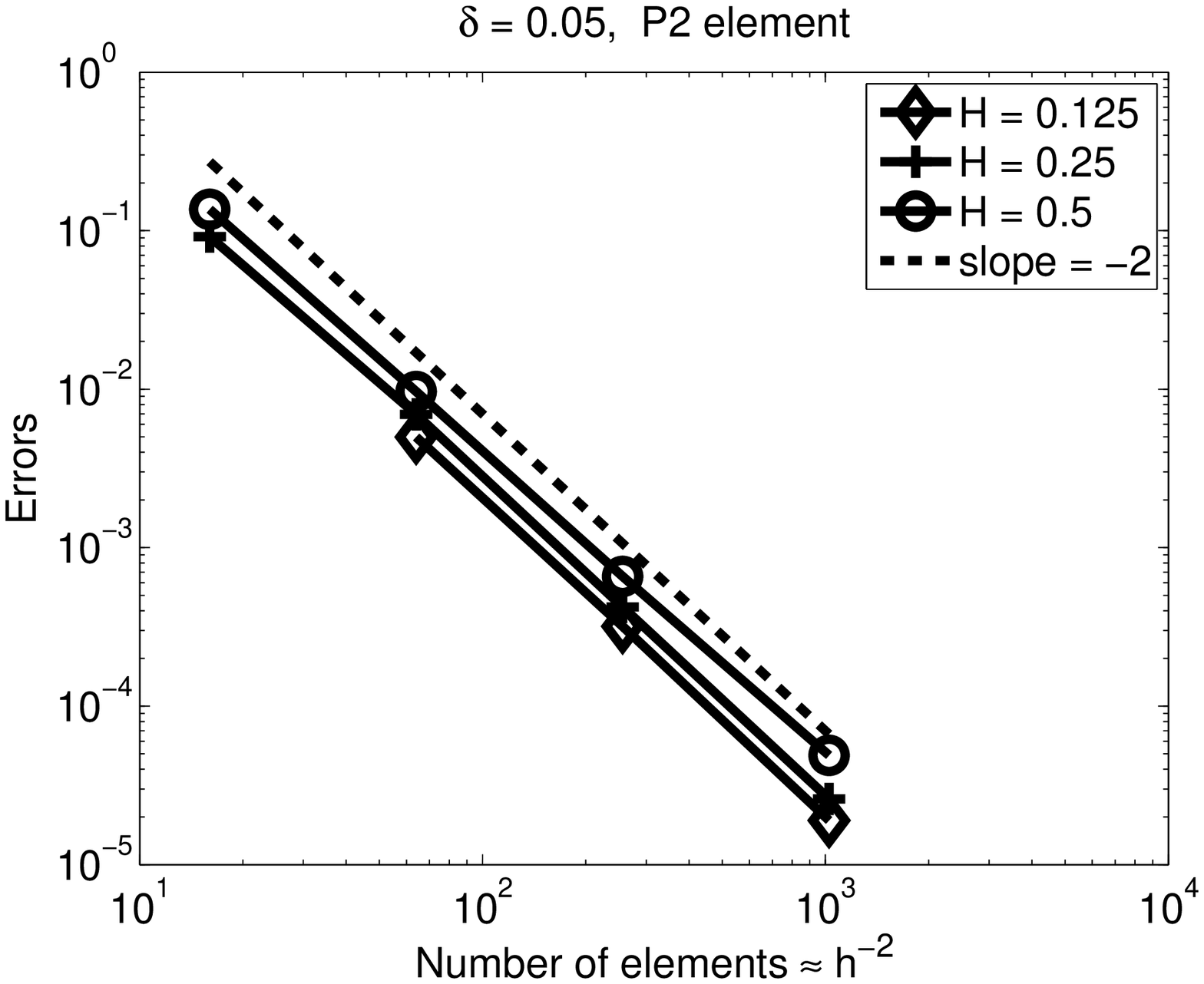}
\caption{The error estimate for the first $5$ eigenvalue approximations by the quadratic
element: The left subfigure is for $H=0.25$ and $\delta=0.05$, $0.1$, $0.2$. The right
 subfigure is for $\delta=0.05$ and $\delta=0.25$, $0.125$, $0.0625$}\label{fig:P2HD}
\end{figure}

\begin{example}\label{Example_2}
In the second example, we solve the eigenvalue problem (\ref{Weak_Eigenvalue_Problem})
using linear and quadratic element on the square $\Omega=(-1,1)\times(-1,1)$
with $a(u,v) = \int_{\Omega}A\nabla u\cdot\nabla v\mathrm{d}\Omega$,
 $b(u,v) = \int_{\Omega}\phi u v\mathrm{d}\Omega$  and
\begin{equation*}
A = \begin{pmatrix} e^{1+x^2} & e^{xy} \\  e^{xy} & e^{1+y^2} \end{pmatrix} \ \
\text{   and   } \ \  \phi = (1+x^2)(1+y^2).
\end{equation*}
\end{example}
Since the exact eigenvalue is not known, we use the accurate enough approximations
$[17.982932, 33.384973, 38.381968, 47.670103, 66.874113, 68.323961]$
by the extrapolation method as the first $6$ exact eigenvalues to investigate the errors.
\begin{figure}[htb]
\centering
\includegraphics[width=5cm,height=5cm]{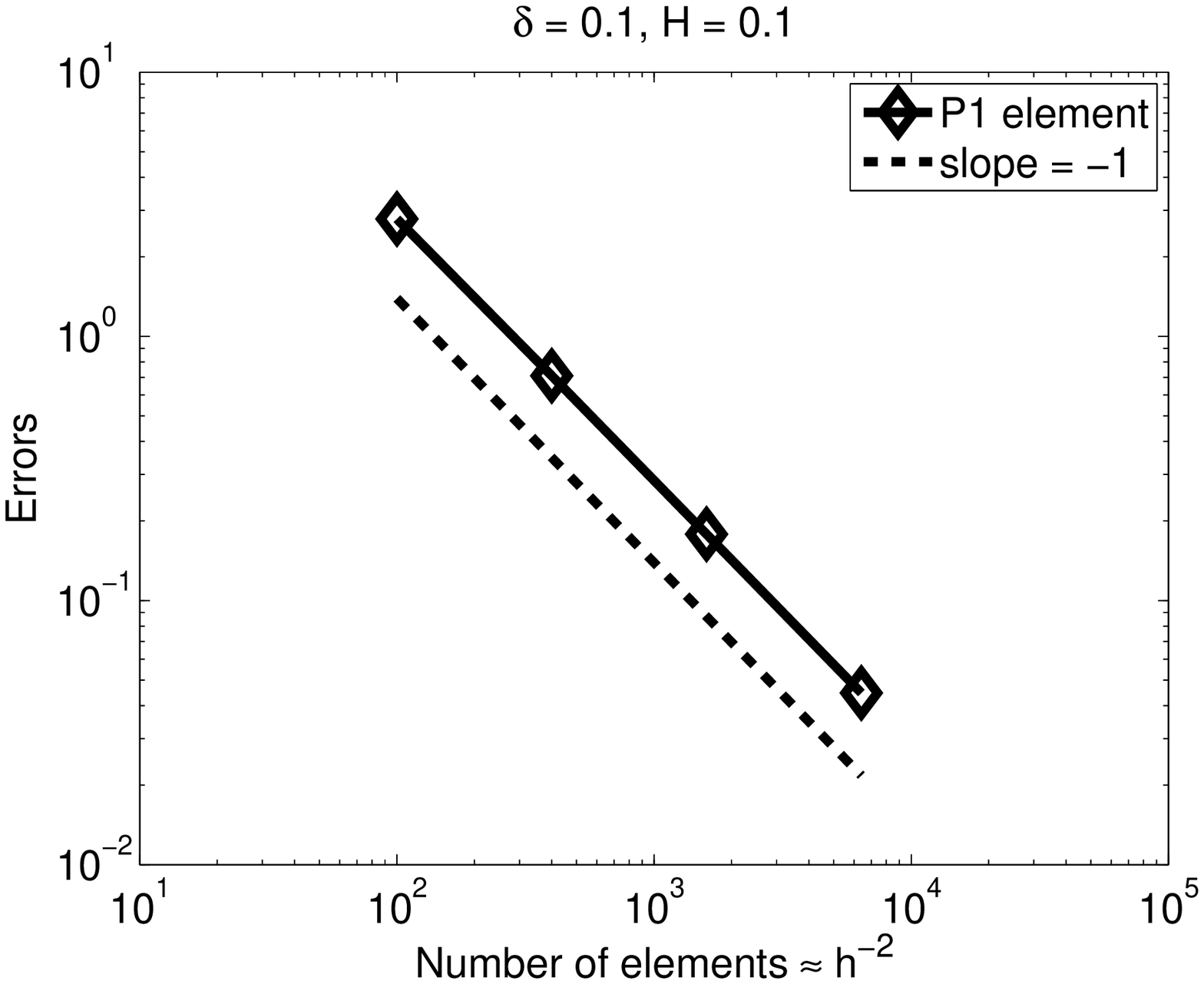}
\includegraphics[width=5cm,height=5cm]{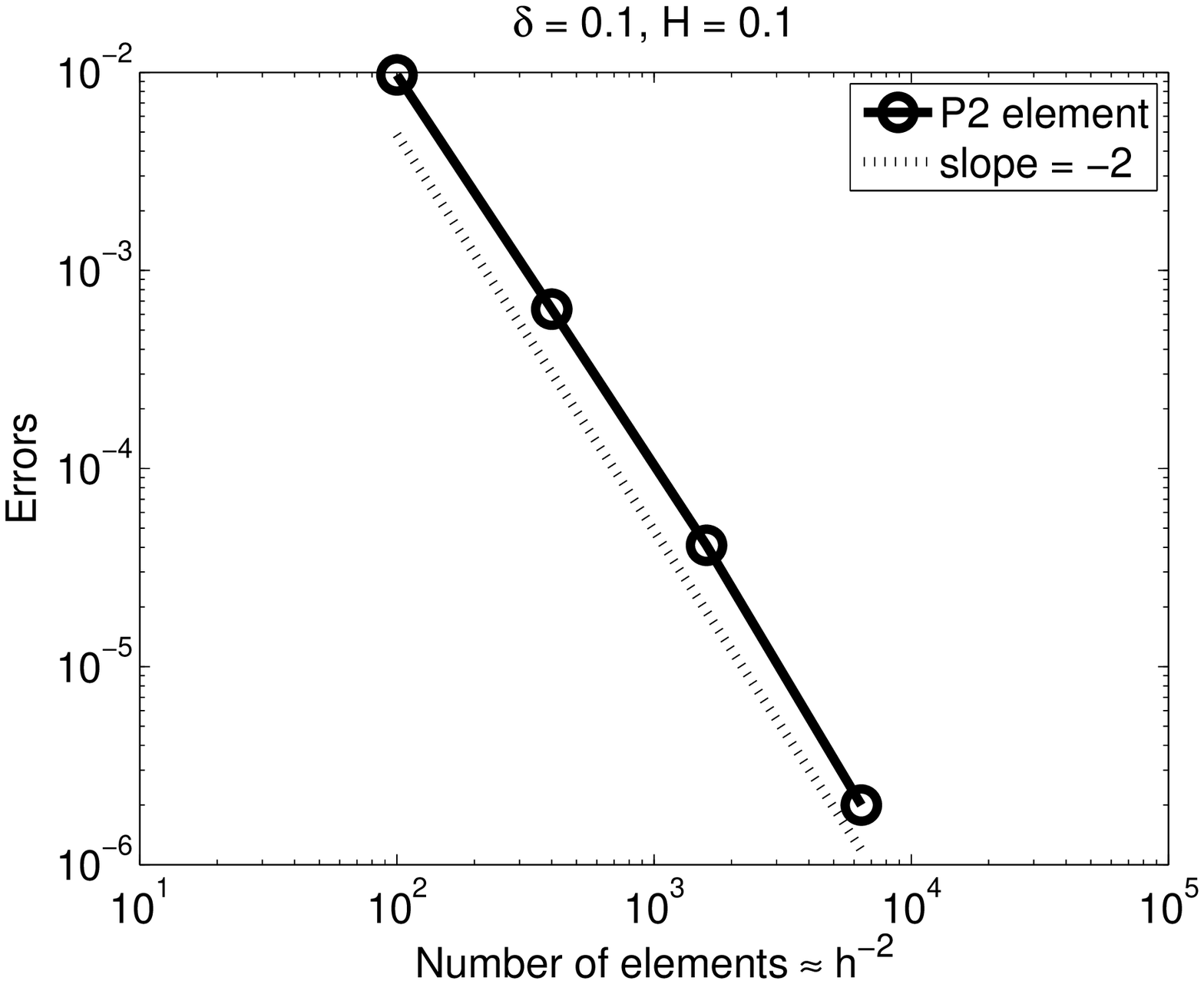}
\caption{The error estimate for the first $6$ eigenvalue approximations with
$H=0.1$ and $\delta=0.1$: The left subfigure is for linear element and the right
 subfigure is for quadratic element}\label{fig:P1P2}
\end{figure}
Figure \ref{fig:P1P2} shows the corresponding numerical results
for the first  $6$ eigenvalues by the linear and quadratic finite element methods, respectively.
Here, we use four level grids to do the numerical experiments. From Figure \ref{fig:P1P2}, the
numerical results also confirm the efficiency of the proposed algorithm in this paper.

\section{Concluding remarks}
In this paper, we give a new type of multilevel local and parallel method based on multigrid discretization
to solve the eigenvalue problems. The idea here is to use the multilevel correction method
to transform the solution of eigenvalue problem to a series of solutions of the corresponding boundary value
 problems with the local and parallel method. As stated in the numerical examples, Algorithm
 \ref{Algm:Multi_Correction} for simple eigenvalue cases can be extended to the corresponding version
for multiple eigenvalue cases. For more information, please refer \cite{Xie_Nonconforming}.

Furthermore, the framework here can also be coupled with the adaptive refinement technique.
The ideas can be extended to other types of linear and nonlinear eigenvalue problems.
These will be investigated in  our future work.

\end{document}